\documentclass{amsart}

\usepackage{graphicx}
\usepackage[latin1]{inputenc}

\newtheorem{thm}{Theorem}[section]

\newtheorem{lem}[thm]{Lemma}

\newtheorem{exa}[thm]{Example}
\newtheorem{dfn}[thm]{Definition}
\newtheorem{rem}[thm]{Remark}
\numberwithin{equation}{section}

\begin{document}

\title[]{Remarks on the invariants valued in the generalization of Conway algebra}

\author{Seongjeong Kim}

\address{Department of Fundamental Sciences, Bauman Moscow State Technical University, Moscow, Russia \\
ksj19891120@gmail.com}



\begin{abstract}
In~\cite{Kim} the author generalized the Conway algebra and constructed the invariant valued in the generalized Conway algebra defined by applying two skein relations to crossings, which is called a generalized Conway type invariant. The generalized Conway type invariant is a generalization of Homflypt polynomial.

In this paper we show that an example of links, which have the same value of Homflypt polynomial, but have different values of the generalized Conway type invariant. We study a properties of Conway type invariant related to Vassiliev invariant. In section 3 we discuss about further researches.
\end{abstract}

\maketitle

\section{Introduction}
In \cite{PrzytyskiTraczyk} J.H.Przytycki and P.Traczyk introduced an algebraic structure, called {\it a Conway algebra,} and constructed an invariant of oriented links, which is a generalization of the Homflypt polynomial invariant. In 2017 L. H. Kauffman and S. Lambropoulou~\cite{KauffmanLambropoulou} constructed new 4-variable polynomial invariants, which are generalized from the Homflypt polynomial, the Dubrovnik polynomial and the Kauffman polynomial. In~\cite{Kim} the author generalized the Conway algebra and constructed the invariant valued in the generalized Conway algebra, which is constructed by applying two skein relations to crossings.

In this paper we show that an example of links, which have the same value of Homflypt polynomial, but have different values of the generalized Conway type invariant. We study a properties of Conway type invariant related to Vassiliev invariant. In section 3 we discuss about further researches.

\section{Some remarks on the Conway type invariant}
We introduce a generalization of the Conway algebra and invariant valued in the generalized Conway algebra.

\begin{dfn}\cite{Kim}\label{def_genConaltype1}
Let $\widehat{\mathcal{A}}$ be a set with four binary operations  $\circ,*,/$ and $//$ on $\widehat{\mathcal{A}}$. Let $\{a_{n}\}_{n=1}^{\infty} \subset \widehat{\mathcal{A}}$. The hexuple $( \widehat{\mathcal{A}}, \circ,/,*, //,\{a_{n}\}_{n=1}^{\infty})$ is called {\it a generalized Conway algebra} if it satisfies the following conditions:
\begin{description}
\item[(A)] $(a \circ b) / b = (a / b) \circ b = a = (a* b) // b  = (a // b) * b$ for $a,b \in  \widehat{\mathcal{A}}$,
\item[(B)] $ a_{n} = a_{n} \circ a_{n+1}$ for $n = 1,2,\cdots,$
\item[(C)] $(a \circ b) \circ (c \circ d) = (a \circ c) \circ (b \circ d)$ for $a,b,c,d \in  \widehat{\mathcal{A}}$,
\item[(D)] $(a * b) * (c * d) = (a * c) * (b * d)$ for $a,b,c,d \in  \widehat{\mathcal{A}}$,
\item[(E)] $(a \circ b) \circ (c * d) = (a \circ c) \circ (b * d)$ for $a,b,c,d \in  \widehat{\mathcal{A}}$,
\item[(F)] $(a * b) * (c \circ d) = (a * c) * (b \circ d)$ for $a,b,c,d \in  \widehat{\mathcal{A}}$,
\item[(G)] $(a \circ b) * (c \circ d) = (a * c) \circ (b * d)$ for $a,b,c,d \in  \widehat{\mathcal{A}}$.
\end{description}

\end{dfn}

\begin{rem}
Let $( \widehat{\mathcal{A}}, \circ,/,*,//,\{a_{n}\}_{n=1}^{\infty})$ be a generalized Conway algebra. If two operations $\circ$ and $*$ are same, then the generalized Conway algebra is a Conway algebra, and hence the Conway type invariant can be defined on $(\widehat{\mathcal{A}}, \circ, /,\{a_{n}\}_{n=1}^{\infty})$.
\end{rem}

\begin{thm}\cite{Kim}\label{Main_thm}
Let $\mathcal{L}$ be the set of equivalence classes of oriented link diagrams modulo Reidemeister moves. Let $(\widehat{\mathcal{A}}, \circ,/,*, //,\{a_{n}\}_{n=1}^{\infty})$ be a generalized Conway algebra. Then there uniquely exists the invariant of classical oriented links $\widehat{W} : \mathcal{L} \rightarrow \widehat{\mathcal{A}}$ satisfying the following rules:
\begin{enumerate}
\item For self crossings $c$ the following relation holds: 
\begin{equation}\label{selfConwayrel}
\widehat{W}(L_{+}^{c}) = \widehat{W}(L_{-}^{c}) \circ \widehat{W}(L_{0}^{c}).
\end{equation}
\item For mixed crossings $c$ the following relation holds: 
\begin{equation}\label{mixedConwayrel}
\widehat{W}(L_{+}^{c}) = \widehat{W}(L_{-}^{c}) * \widehat{W}(L_{0}^{c}).
\end{equation}
\item Let $T_{n}$ be a trivial link of $n$ components. Then
\begin{equation}
\widehat{W}(T_{n}) = a_{n}.
\end{equation}
\end{enumerate}
We call $\widehat{W}$ {\it a generalized Conway type invariant valued in $(\widehat{\mathcal{A}}, \circ,/,*,//,\{a_{n}\}_{n=1}^{\infty})$.}
\end{thm}

\textbf{Construction of $\widehat{W}$.} First we will define $\widehat{W}$ for ordered oriented link diagrams. Let $L = L_{1} \cup \cdots \cup L_{r}$ be an ordered oriented link diagram of $r$ components. Fix a base point $b_{i}$ on each component $L_{i}$. Suppose that we walk along the diagram $L_{1}$ according to the orientation from the base point $b_{1}$ to itself, then we walk along the diagram $L_{2}$ from the base point $b_{2}$ to itself and so on. If we pass a crossing $c$ first along the undercrossing(or overcrossing), we call $c$ {\it a bad crossing}(or {\it a good crossing}) with respect to the base points $b = \{b_{1}, \cdots, b_{r}\}$. Now we perform the crossing change or splicing for all bad crossings. Denote the value of  $\widehat{W}$ for $L$ corresponding to the base points $b$ by $\widehat{W}_{b}(L)$. Suppose that we meet the first bad crossing $c$. If it is a self crossing, we apply the skein relation on $c$ with the following property:
\begin{equation*}
\widehat{W}_{b}(L_{+}^{c}) = \widehat{W}_{b}(L_{-}^{c}) \circ \widehat{W}_{\tilde{b}}(L_{0}^{c}),
\end{equation*}
where $\tilde{b} = \{b_{1}, \cdots, b_{j-1},b_{j},b'_{j},b_{j+1}, \cdots, b_{n},b_{n+1}\}$ and $b'_{j}$ is a chosen base point near the place of the crossing $c$ on the component, which is appeared by splicing the self crossing $c$ of $L_{j}$, see~Fig.~\ref{1-2splicing}.
\begin{figure}[h!]
\begin{center}
 \includegraphics[width = 8cm]{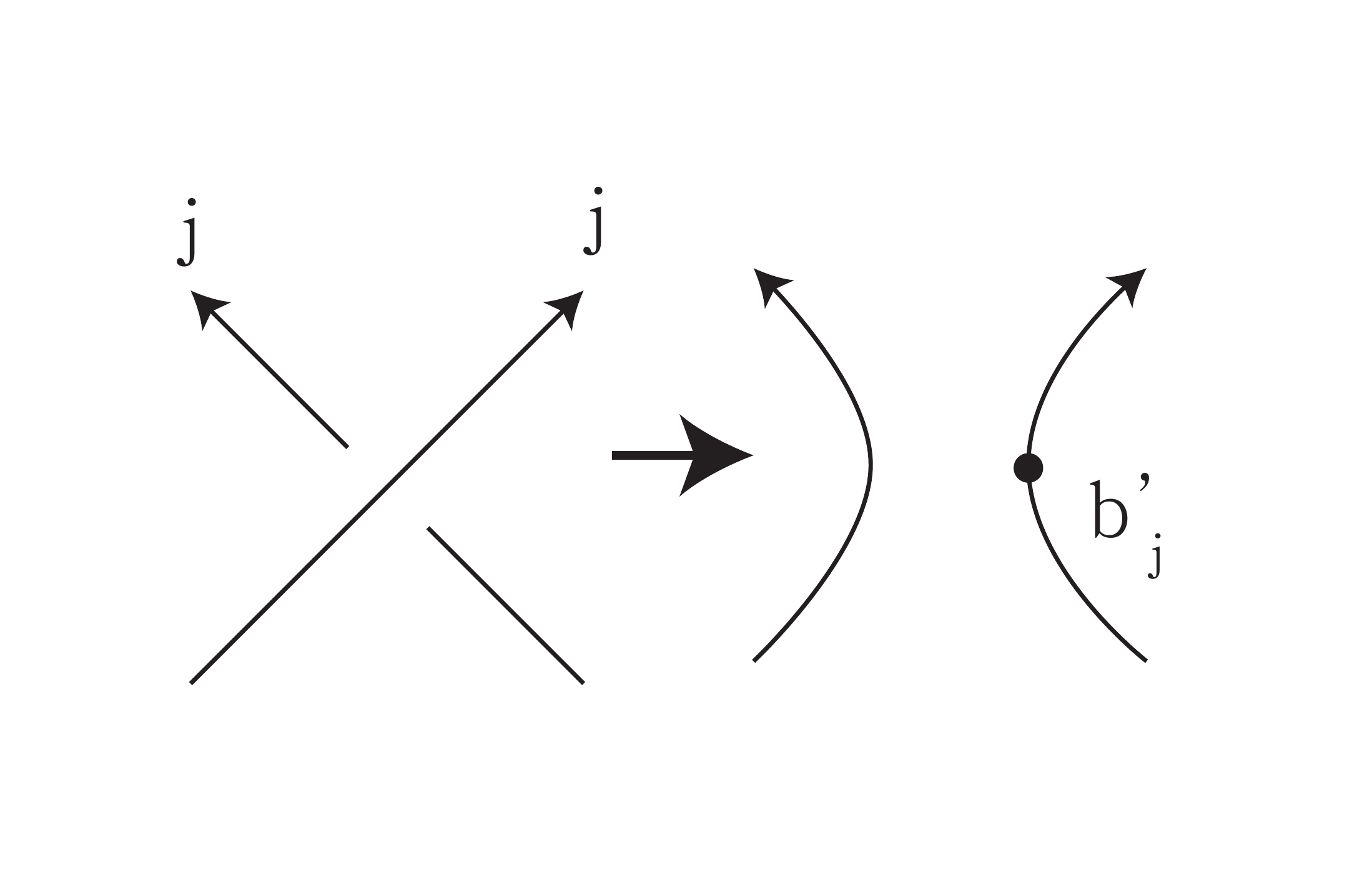}

\end{center}

 \caption{The choice of a base point $b'_{j}$}\label{1-2splicing}
\end{figure}

If it is a mixed crossing between two components $L_{i}$ and $L_{j}$, we apply the skein relation on $c$ with the following property:
\begin{equation*}
\widehat{W}_{b}(L_{+}^{c}) = \widehat{W}_{b}(L_{-}^{c}) *\widehat{W}_{\tilde{b}}(L_{0}^{c}),
\end{equation*}
where $\tilde{b} = \{b_{1}, \cdots, b_{j-1},b_{j+1}, \cdots, b_{n}\}$.
As an abuse of notation we write $\tilde{b} = b$, if it does not cause confusion. Notice that if $c$ is a positive crossing, then the number of bad crossings of $L_{-}^{c}$ is less than the number of bad crossings of $L_{+}^{c}$ and the number of crossings of $L_{0}^{c}$ is less than the number of crossings of $L_{+}^{c}$. We apply those relations to the first bad crossings of $L_{-}^{c}$ and $L_{0}^{c}$ inductively until we switch all bad crossings. If $c$ is a negative crossing, then we apply those relations to the first bad crossings of $L_{+}^{c}$ and $L_{0}^{c}$ inductively until we switch all bad crossings.
 If $L = L_{1} \cup \cdots \cup L_{r}$ has no bad crossings, then we define $\widehat{W}_{b}(L) = a_{n}$.

\begin{exa}\label{exa_gen_Conway}
Let $\widehat{\mathcal{A}} = \mathbb{Z} [p^{\pm 1 }, q^{\pm 1}, r]$. Define the binary operations $\circ,*,/$ and $//$ by
\begin{eqnarray*}
a \circ b = pa + qb,& a / b = p^{-1} a - p^{-1}qb,\\
a * b = pa + rb,& a // b = p^{-1} a - p^{-1}rb.
\end{eqnarray*}
Denote $a_{n} = (\frac{1-p}{q})^{n-1}$ for each $n$. Then $(\widehat{\mathcal{A}}, \circ,/,*, //,\{a_{n}\}_{n=1}^{\infty})$ is a generalized Conway algebra.
\end{exa}

\begin{exa}\label{exa_gen_Homflypt}
Let $\widehat{\mathcal{A}} = Z[v^{\pm 1}, z, w^{\pm 1}]$ be an algebra. Define binary operations $\circ,/,*$ and $//$ by
\begin{eqnarray*}
a\circ b = v^{2}a+vwb,& a/b = v^{-2}a-v^{-1}wb,\\ 
a* b = v^{2}a+vzb,& a//b = v^{-2}a-v^{-1}zb.
\end{eqnarray*}
Put $a_{n} = (\frac{v^{-1}-v}{w})^{n-1}.$ Then $(\widehat{\mathcal{A}},\circ,/,*,//,\{a_{n}\}_{n=1}^{\infty})$ is a generalized Conway algebra. In fact, this is obtained from the generalized Conway algebra in example~\ref{exa_gen_Conway} by substituting $p=v^{2}$, $q = vw$ and $r = vz$. Moreover if $w=z$, then the Conway type invariant valued in the generalized Conway algebra $(\widehat{\mathcal{A}},\circ,/,*,//,\{a_{n}\}_{n=1}^{\infty})$ is the Homflypt polynomial.
\end{exa}

\begin{exa}\
Let $(\widehat{\mathcal{A}},\circ,/,*,//,\{a_{n}\}_{n=1}^{\infty})$ be the generalized Conway algebra in the previous example. It is well known that two links $L_{1} = L11n418\{0,0\}$ and $L_{2} = L11n358\{0,1\}$ have the same value of Homflypt polynomial. The values of Conway type invariant valued in $(\widehat{\mathcal{A}},\circ,/,*,//,\{a_{n}\}_{n=1}^{\infty})$ from $L_{1}$ and $L_{2}$ are calculated as follow.
\begin{eqnarray*} 
\widehat{W}(L_{1}) &=& \frac{1}{p^3 q^2}-\frac{2}{p^2 q^2}+\frac{1}{p q^2}+\frac{r}{q}+\frac{r}{p^4 q}
-\frac{6 r}{p^3 q}+\frac{7 r}{p^2 q}-\frac{3 r}{p q}-\frac{4r^2}{p^4}+\frac{9 r^2}{p^3}\\
&&-\frac{4 r^2}{p^2}+\frac{r^2}{p}+\frac{q r^2}{p^3}-\frac{q r^3}{p^5}+\frac{5 q r^3}{p^4}-\frac{2 q r^3}{p^3}+\frac{q^2
r^4}{p^5}
\end{eqnarray*}
\begin{eqnarray*} \widehat{W}(L_{2}) &=&\frac{1}{p^3 q^2}-\frac{2}{p^2 q^2}+\frac{1}{p q^2}-\frac{r}{p^5 q}+\frac{4 r}{p^4 q}-\frac{8 r}{p^3 q}+\frac{5 r}{p^2 q}\\
&&+\frac{3 r^2}{p^5}-\frac{8r^2}{p^4}+\frac{8 r^2}{p^3}-\frac{r^2}{p^2}-\frac{3 q r^3}{p^5}+\frac{4 q r^3}{p^4}-\frac{q r^3}{p^3}+\frac{q^2 r^4}{p^5}
\end{eqnarray*}

\begin{figure}[h!]
\begin{center}
 \includegraphics[width = 8cm]{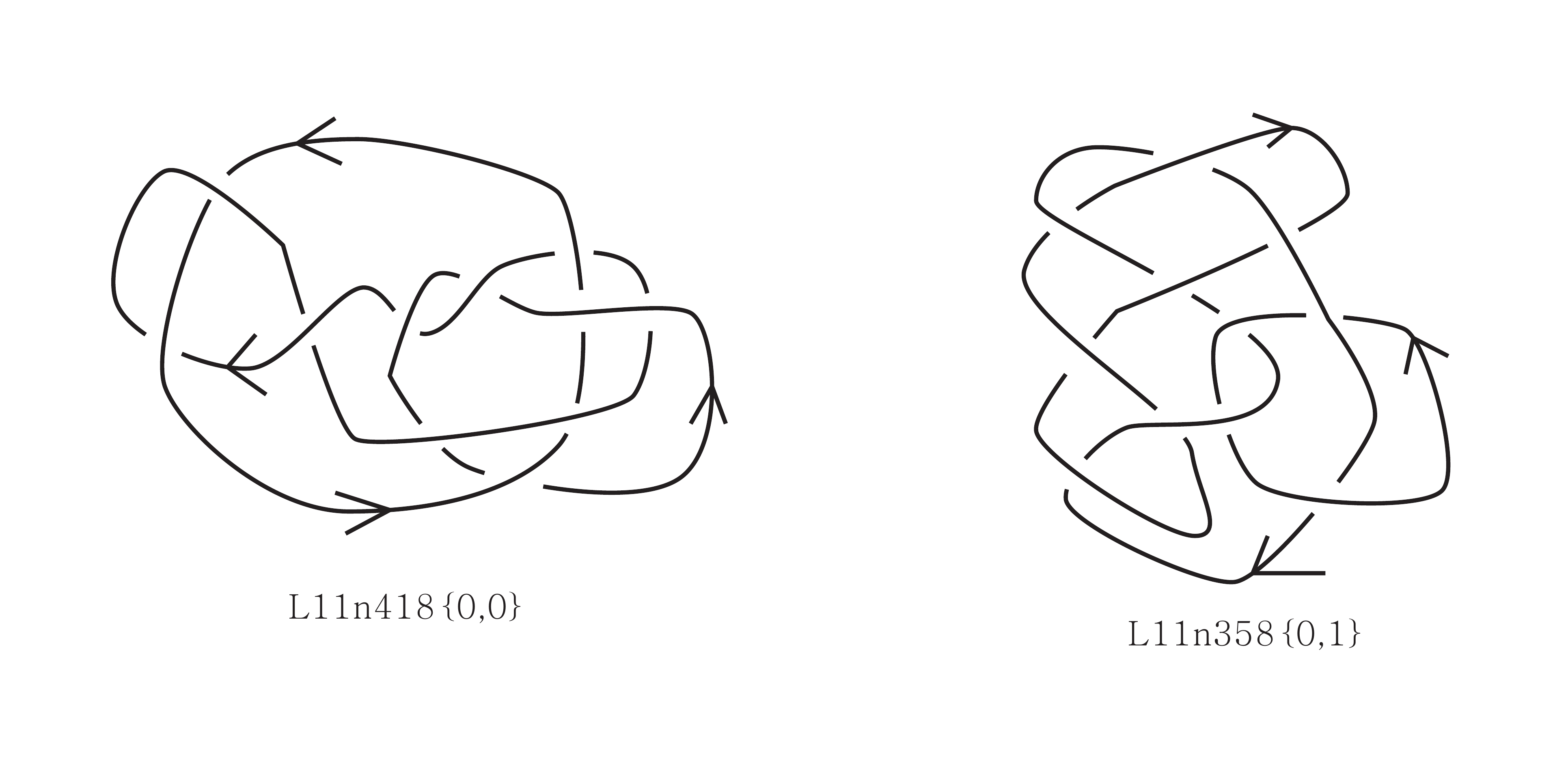}

\end{center}

 \caption{Diagrams $L_{1} = L11n418\{0,0\}$ and $L_{2} = L11n358\{0,1\}$}\label{exa}

\end{figure}
That is, $\widehat{W}(L_{1}) \neq \widehat{W}(L_{2})$ and the generalized Conway type invariant is stronger than Homflypt polynomial.
\end{exa}
\begin{exa}\label{exa_non_linear_gen_Conway}
Let us fix a natural number $k$. Let $\widehat{\mathcal{A}}$ be the smallest commutative ring with identity containing $\mathbb{Z}[p^{\pm1},q^{\pm1},r^{\pm1}]$ such that $\sqrt[k]{f} \in \widehat{\mathcal{A}}$ for each $f \in \widehat{\mathcal{A}}$, where $\sqrt[k]{f}$ is the formal $k$-th root, that is, $(\sqrt[k]{f})^{k} = \sqrt[k]{f^{k}} =f$. Define binary operations $\circ,/,*$ and $//$ by
\begin{eqnarray*}
a\circ b = \sqrt[k]{pa^{k}+qb^{k}}, & a/b =  \sqrt[k]{p^{-1}a^{k}-p^{-1}qb^{k}} \\
a * b = \sqrt[k]{pa^{k}+rb^{k}}, & a//b =  \sqrt[k]{p^{-1}a^{k}-p^{-1}rb^{k}},
\end{eqnarray*}
for $a,b \in \widehat{\mathcal{A}}$. Let $\{a_{n}\}$ be the sequence defined by the following recurrence relation
$$ a_{1} = 1, a_{n+1}^{k} = \frac{(1-p)}{q}a_{n}^{k}.$$
We can show that $(\widehat{\mathcal{A}}, \circ,/,*, //,\{a_{n}\}_{n=1}^{\infty})$ is a generalized Conway algebra.
\end{exa}







\begin{lem}
The generalized Conway type invariant valued in the generalized Conway algebra in example~\ref{exa_gen_Homflypt} weaker than two variable Vassiliev invariant.
\end{lem}

\begin{proof}
Let $\widehat{\mathcal{A}} = Z[v^{\pm 1}, z, w^{\pm 1}]$ be an algebra with binary operations $\circ,/,*$ and $//$  defined by
\begin{eqnarray*}
a\circ b = v^{2}a+vwb,& a/b = v^{-2}a-v^{-1}wb,\\ 
a* b = v^{2}a+vzb,& a//b = v^{-2}a-v^{-1}zb.
\end{eqnarray*}
and $a_{n} = (\frac{v^{-1}-v}{w})^{n-1}.$
By substituting $v = e^{x}e^{y}e^{z}$, $w=e^{-x}-e^{x}$ and $z=e^{-y}-e^{y}$ we obtain that
$$\widehat{W}(L^{C}_{+}) - \widehat{W}(L^{C}_{-}) = xf(x,y,z),$$
for a self crossing $C$ and
$$ \widehat{W}(L^{C}_{+}) - \widehat{W}(L^{C}_{-}) = yg(x,y,z)$$
for a mixed crossing $C$, 
therefore the generalized Conway type invariant is weaker than two variable Vassiliev invariant.
\end{proof}
\section{Further research}
When we define the Conway type invariant, we applied two different skein relations to self crossings and mixed crossings.	
It is well known that the division of crossings to {\sl self/mixed crossings} is a parity, which is introduced by V.O.Manturov in \cite{Ma}. We would like to generalized Conway type invariant by applying two different skein relations to even crossings and odd crossings. On the above purpose we define {\bf the link parity} as follow.
\begin{dfn}
A link parity $p$ on diagrams of a link $\mathcal{L}$ with coefficients in $\mathbb{Z}_{2}$ is a family of maps $p_{L}: \mathcal{V}(L) \rightarrow A,$ $L \in ob(\mathcal{L})$ is an object of the category, such that for any elementary morphism $f : L \rightarrow L'$ the following hols:
\begin{enumerate}
\item $p_{L'}(f_{*}(v)) =p_{L}(v)$ provided that $v \in  \mathcal{V}(L) $ and there exists $f_{*}(v) \in  \mathcal{V}(L')$, if $f$ is a Reidemeister moves;
\item $p_{L}(v_{1}) + p_{L}(v_{2}) = 0$ if f is a decreasing second Reidemeister move and $v_{1}, v_{2}$ are the disappearing crossings;
\item $p_{L}(v_{1}) + p_{L}(v_{2}) + p_{L}(v_{3}) = 0$ if $f$ is a third Reidemeister move and $v_{1}, v_{2}, v_{3}$ are the crossing participating in this move;
\item $p_{L}(f_{*}(v)) = p_{L}(v)$ if $f$ is a splicing of $v'$ and $p_{L}(v) + p_{L}(v') = 1$;
\item If $f$ is a splicing of $v'$, $p_{L}(v) + p_{L}(v') = 0$ and $p_{L}(f_{*}(v)) + p_{L}(v) =1$, then $p_{L}(f'_{*}(v')) + p_{L}(v') =1$, where $f'$ is a splicing of $v$.
 
\end{enumerate}

\end{dfn}
It is easy to show that the division of crossings to self/mixed crossings satisfies the above conditions. But the (link) parity, which gives non-trivial parity for knots, is not known. 
Now we define the descending diagram for virtual link diagrams.

\begin{dfn}
Let $L = L_{1} \cup \cdots \cup L_{n}$ be an ordered oriented virtual link diagram of $n$ components. Fix a point $b_{i}$ on each component $L_{i}$, but it is not a crossings. Now we walk along the component $L_{1}$ from the fixed point $b_{1}$ in the given direction of $L_{1}$, and then we walk along the component $L_{2}$ from the fixed point $b_{2}$ in the given direction of $L_{2}$ and so on. If while walking along the diagram each classical crossing is first passed over and then under, then we call the diagram $L$ {\it descending diagram with respect to $\{b_{i}\}_{i=1}^{n}$ and the order of components}. For an oriented virtual link diagram $L$, if we can give an order of components and fix a point on each component of $L$ for $L$ to be a descending diagram, then we call $L$ is a descending diagram. It is well known that if $L$ is a classical descending link diagram, then it is equivalent to the trivial link diagram.
\end{dfn} 

We would like to define {\it the virtual Conway algebra $(\widetilde{\mathcal{A}}, \circ,/,*, //,\{a_{T}\}_{T \in I} \subset \widetilde{\mathcal{A}})$,} which is a generalization of the generalized Conway algebra $(\widehat{\mathcal{A}}, \circ,/,*, //,\{a_{n}\}_{n \in \mathbb{N}})$ to define invariant for oriented virtual links as follow: \\
Let $\widetilde{\mathcal{L}}$ be the set of equivalence classes of oriented virtual link diagrams modulo generalized Reidemeister moves such that every diagram is equipped by link parity. Let $(\widetilde{\mathcal{A}}, \circ,/,*, //,\{a_{T}\}_{T \in I} \subset \widetilde{\mathcal{A}})$ be a virtual Conway algebra, where $I$ is the set of equivalence classes of all descending virtual link diagrams. We would like to define the invariant of oriented virtual links $\widetilde{W} : \widetilde{\mathcal{L}} \rightarrow \widetilde{\mathcal{A}}$ satisfying the following rules:
\begin{enumerate}
\item For even crossings $c$ the following relation holds: 
\begin{equation}\label{selfConwayrel}
\widetilde{W}(L_{+}^{c}) = \widetilde{W}(L_{-}^{c}) \circ \widetilde{W}(L_{0}^{c}).
\end{equation}
\item For odd crossings $c$ the following relation holds: 
\begin{equation}\label{mixedConwayrel}
\widetilde{W}(L_{+}^{c}) = \widetilde{W}(L_{-}^{c}) * \widetilde{W}(L_{0}^{c}).
\end{equation}
\item Let $L_{T}$ be a virtual link, which has a descending diagram $T$. Then
\begin{equation}
\widetilde{W}(L_{T}) = a_{T}.
\end{equation}
\end{enumerate}
But when we define the virtual Conway algebra, we meet the following difficulty: it is well-known that there are infinitely many nontrivial free links. It follows that there are infinitely many nontrivial descending virtual link diagrams, because by forgetting under/over information and the underlying surface of virtual link diagram, we obtain free link diagram. But we don't know how to list all such virtual link diagrams. Moreover, to define the virtual Conway algebra $(\widetilde{\mathcal{A}}, \circ,/,*, //,\{a_{T}\}_{T \in I} \subset \widetilde{\mathcal{A}})$ we need to find relations 
\begin{equation}\label{needed_rels}
a_{T_{+}} = a_{T_{-}} \circ a_{T_{0}}, a_{T_{+}} = a_{T_{-}} * a_{T_{0}}
\end{equation}

for $T_{*} \in I$, where $T_{+},T_{-}$ and $T_{0}$ are the Conway triple.\\

\textbf{Questions}\\
\begin{enumerate}
\item Is there a (link) parity, which gives non-trivial parity for classical links?
\item How to find the set $\{a_{T}\}_{T \in I}$ and the relations~\ref{needed_rels} for the invariant $\widetilde{W}$ is well-defined.
\end{enumerate}

\end{document}